\documentclass[12pt,reqno]{article}
\usepackage[letterpaper,top=1.9cm, bottom=1.9cm, left=1.75cm, right=1.75cm]{geometry}
\usepackage{amssymb}
\usepackage{amsmath}
\usepackage{amsfonts}
\usepackage{amsthm}
\usepackage{graphicx}
 
\usepackage{latexsym} 
\usepackage{bm}
\usepackage{bbm}
\usepackage[normalem]{ulem} 
\usepackage{overpic}
\usepackage{mathrsfs}
 \usepackage{enumerate}
\usepackage{hyperref}
\hypersetup{
    colorlinks,
    citecolor=black,
    filecolor=black,
    linkcolor=blue,
    urlcolor=black
}

\def\red#1{\textcolor{red}{#1}}

\newtheorem{theorem}{Theorem}[section]

\newtheorem{lemma}[theorem]{Lemma}
\newtheorem{corollary}[theorem]{Corollary}

\theoremstyle{definition}

\newtheorem{remark}[theorem]{Remark}

\def\Ind#1{{\mathbbmss 1}_{_{\scriptstyle #1}}}

\def\<{\langle}
\def\>{\rangle}
\def\wt#1{\widetilde{#1}}

\renewcommand{\baselinestretch}{1}\normalsize

\title{ On (signed) Takagi--Landsberg functions:\\ $p^{\text{th}}$ variation, maximum, and modulus of continuity}
\author{\normalsize Yuliya Mishura\thanks{This author acknowledges  that the present research is partially supported by the ToppForsk project nr.~274410 of the Research Council of Norway with title STORM: Stochastics for Time-Space Risk Models.}\\
\normalsize Department of Probability, Statistics\\
\normalsize  and Actuarial Mathematics\\ \normalsize Taras Shevchenko National University of Kyiv\\
\normalsize  01601 Kyiv, Ukraine\\
\normalsize    {\tt myus@univ.kiev.ua}\and  \setcounter{footnote}{6}\normalsize Alexander Schied\thanks{
Support from the
 Natural Sciences and Engineering Research Council of Canada through grant RGPIN-2017-04054 is gratefully acknowledged.}\\  \normalsize Dept.~of Statistics and Actuarial Science\\
 \normalsize University of Waterloo\\  \normalsize Waterloo, Ontario, N2L 3G1, Canada
\\
 \normalsize {\tt aschied@uwaterloo.ca} }

\begin{document}
\date{\small  First version: June 14, 2018\\
\small  This version: June 19, 2019} 
\maketitle
\begin{abstract}
We study a class $\mathfrak X^H$ of signed Takagi--Landsberg functions with Hurst parameter $H\in(0,1)$. We first show that the functions in $\mathfrak X^H$ admit a linear $p^{\text{th}}$ variation along the  sequence of dyadic partitions of $[0,1]$, where $p=1/H$. The slope of the linear increase can be represented as the $p^{\text{th}}$ absolute moment of the infinite Bernoulli convolution with parameter $2^{H-1}$. The existence of a continuous  $p^{\text{th}}$ variation enables the use of the functions in $\mathfrak X^H$ as test integrators for higher-order pathwise It\^o calculus. Our next results concern   the  maximum, the maximizers, and the  modulus of continuity of  the classical Takagi--Landsberg function  for all $0<H<1$. Then we identify the uniform maximum, the uniform maximal oscillation, and a uniform modulus of continuity for the class $\mathfrak X^H$. 
\end{abstract}



\section{Introduction}

The purpose of this paper is twofold. Our first goal is to study geometric properties of (signed) Takagi--Landsberg functions with Hurst parameter $H\in(0,1)$. These are fractal functions that have been widely studied in the literature; see, e.g., the surveys~\cite{AllaartKawamura,GalkinGalkina,Lagarias}. The properties in which we are interested   include size and location of the (uniform) maximum, the uniform maximal oscillation, and the (uniform) modulus of continuity. 
For instance, we derive an explicit formula for   the maximum of the Takagi--Landsberg function $x^H$ and  the location of its maximizers. We prove in particular that   $t =\frac13$ and $t=\frac23$ are the unique points at which the function $x^H(t)$ attains its maximum. In particular, the maximizers are independent of $H$ as long as $H<1$. This pattern changes at $H=1$, where it was shown by Kahane~\cite{Kahane} that  the classical Takagi function attains its maximum at an uncountable Cantor-type set of Hausdorff dimension $\frac12$. 
We then use our result on the maximum of $x^H$ to derive an exact modulus of continuity of the Takagi--Landsberg function. Here again, we observe that our result breaks down at $H=1$ as can be seen from the work of K\^ono~\cite{Kono}.

Our second goal is to establish the class of signed Takagi--Landsberg functions as a natural class of \lq\lq rough test integrators" for a higher-order pathwise integration theory in the spirit of F\"ollmer's pathwise It\^o calculus~\cite{FoellmerIto}. Such an integration theory was recently developed by Cont and Perkowski~\cite{ContPerkowski}; see also Gradinaru et al.~\cite{GradinaruEtAl} and Errami and Russo~\cite{ErramiRusso} for related earlier work.  To this end, we will argue that, for a given Hurst parameter $H\in(0,1)$ and $p:=1/H$, the corresponding signed Takagi--Landsberg functions admit a linear $p^{\text{th}}$ variation along the dyadic partitions of $[0,1]$ as defined in~\cite{ContPerkowski}. The slope of the linear increase of the $p^{\text{th}}$ variation is equal to the expectation $\mathbb E[|Z_H|^p]$, where  $Z_H$  is a random variable whose law is the infinite Bernoulli convolution with parameter $2^{H-1}$.

A high-order pathwise integration theory as mentioned above makes the techniques of model-free finance available for tackling  continuous-time phenomena that are rougher than the usual diffusion paths. One particularly exciting new development is provided by the observation by Gatheral et al.~\cite{GatheralRosenbaum} that \lq\lq volatility is rough". That is, empirical volatility time series exhibit a Hurst parameter much smaller than $1/2$. So far, modeling of rough volatility has been based on fractional Brownian motion (see, e.g.,~\cite{EuchEtAl} and the references therein). With higher-order pathwise It\^o calculus at hand, model-free techniques in the spirit of Bick and Willinger~\cite{BickWillinger}, F\"ollmer~\cite{FoellmerECM}, and the substantial follow-up literature  become feasible. In analogy with our previous papers~\cite{MishuraSchied,SchiedTakagi}, our contribution is to provide an explicit class of test integrators for such model-free approaches.

This paper is organized as follows. 
In Section~\ref{pth variation section}, we state our result on the $p^{\text{th}}$ variation of functions in  our class $\mathfrak X^H$ of signed Takagi--Landsberg functions. 
In Section~\ref{TL function section}, we present several  properties of the Takagi--Landsberg function $x^H$ with Hurst parameter $H\in(0,1)$ and related uniform properties of the class $\mathfrak X^H$. In particular, we derive the (uniform) maximum and its location, the maximal uniform oscillation, and  (uniform) moduli of continuity for $x^H$ and the functions in $\mathfrak X^H$. All proofs are contained in Section~\ref{Proofs}.

\section{Main results}\label{results section}

Recall that the Faber--Schauder functions are defined as 
$$e_\emptyset(t):=t,\qquad e_{0,0}(t):=(\min\{t,1-t\})^+,\qquad e_{n,k}(t):=2^{-n/2}e_{0,0}(2^n t-k)
$$
for $t\in\mathbb{R}$, $n=1,2,\dots$, and $k\in\mathbb{Z}$. It is well known that the restrictions of  the Faber--Schauder functions to $[0,1]$ form a Schauder basis for $C[0,1]$. Conversely, the Faber--Schauder functions can be used to construct some interesting functions. A prominent example is the \emph{Takagi--Landsberg function} with Hurst parameter $H>0$,
\begin{equation}\label{TL fctn eq}
x^H(t):=\sum_{m=0}^{\infty}2^{m(1/2-H)}\sum_{k=0}^{2^{m}-1}  e_{m,k}(t),\qquad 0\le t\le 1.
\end{equation}
 See Figure~\ref{TL function various H} for a plot of $x^H$ for various choices of $H$. 
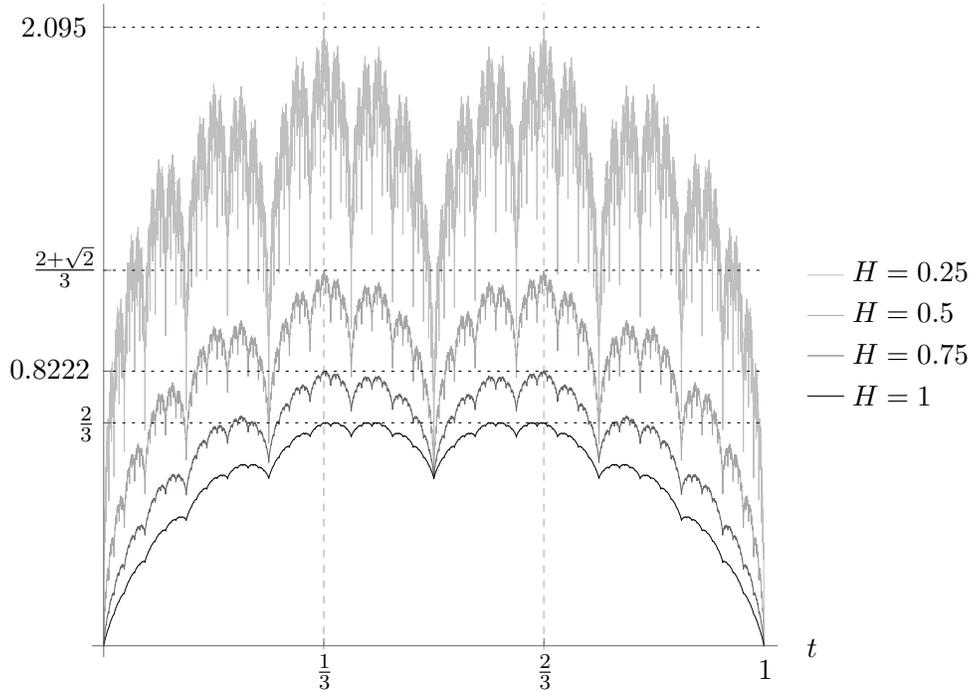
\begin{figure}
\begin{center}
\begin{overpic}[width=10cm]{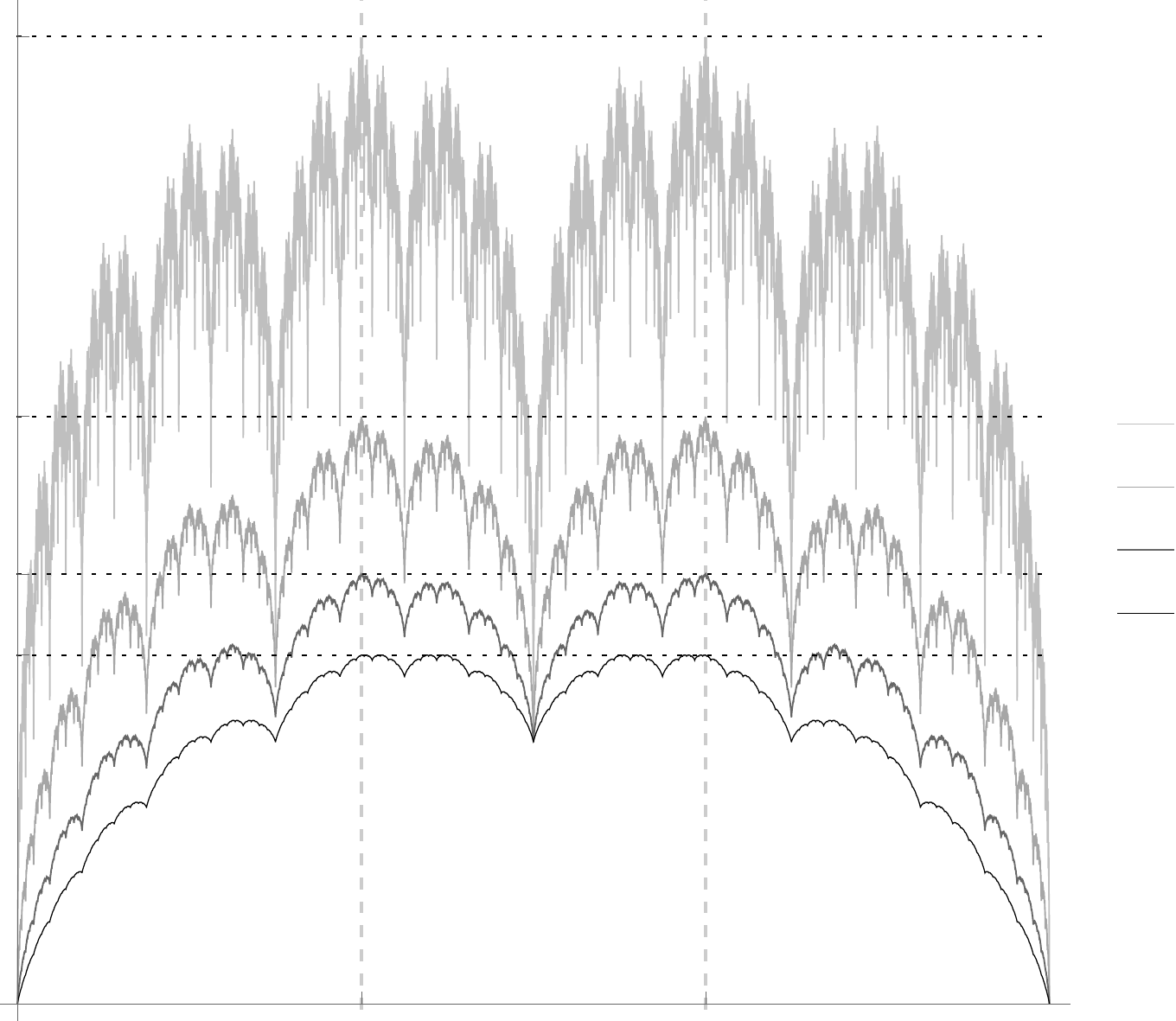}
{\small\put(29.7,-3){$\frac13$}
\put(59,-3){$\frac23$}
\put(88.5,-3){$\tiny1$}
\put(95,0){$t$}
\put(-2,30){$\frac23$}
\put(-11,37){$0.8222$}
\put(-8,50.5){$\frac{2+\sqrt2}3$}
\put(-9.5,82.5){2.095}
\put(101,33.5){$H=1$}
\put(101,39){$H=0.75$}
\put(101,44.5){$H=0.5$}
\put(101,50){$H=0.25$}}
\end{overpic}
\end{center}
\caption{Takagi--Landsberg functions $x^H$ and their maxima for various choices of the Hurst parameter $H$.}\label{TL function various H}
\end{figure}
The case $H=1$ corresponds to the celebrated Takagi function, which was introduced in~\cite{Takagi}, became rediscovered many times, and is sometimes also called the blancmange curve;  see the surveys~\cite{AllaartKawamura,GalkinGalkina,Lagarias}. The extension of this function to a general Hurst parameter was attributed to Landsberg~\cite{Landsberg} by Mandelbrot~\cite[p.~246]{BarnsleyEtal}.
In this paper, we focus on the case $0<H<1$.

We will also study  the function class  that arises if we allow for additional coefficients $\theta_{m,k}\in\{-1,+1\}$ in front of the functions $e_{m,k}$ in \eqref{TL fctn eq}.  We are thus interested in the class
\begin{equation}\label{XH}
\mathfrak{X}^{H}=\left\{x\in C[0,1]\,\bigg|\,x=\sum\limits_{m=0}^{\infty}2^{m\left(\frac{1}{2}-H\right)}\sum\limits_{k=0}^{2^m-1}\theta_{m,k}e_{m,k}\;\;\;  \text{for coefficients}  \;\;\; \theta_{m,k}\in\left\{-1,+1\right\} \right\}
\end{equation}
of \emph{signed Takagi--Landsberg functions} with Hurst parameter $H\in(0,1)$; see Figure~\ref{XH illustration fig} for an illustration. It can be checked easily that the Faber--Schauder series in the definition of the class $\mathfrak{X}^{H}$ converges uniformly for every $H$ and all possible choices $\theta_{m,k}\in\left\{-1,+1\right\}$. 
The class $\mathfrak{X}^{H}$ is a subset of the \emph{flexible Takagi class} introduced by Allaart~\cite{Allaartflexible}, and the special case $ \mathfrak X^{1/2}$ was analyzed in~\cite{SchiedTakagi}. 
\begin{figure}
\begin{center}
\begin{overpic}[width=14cm]{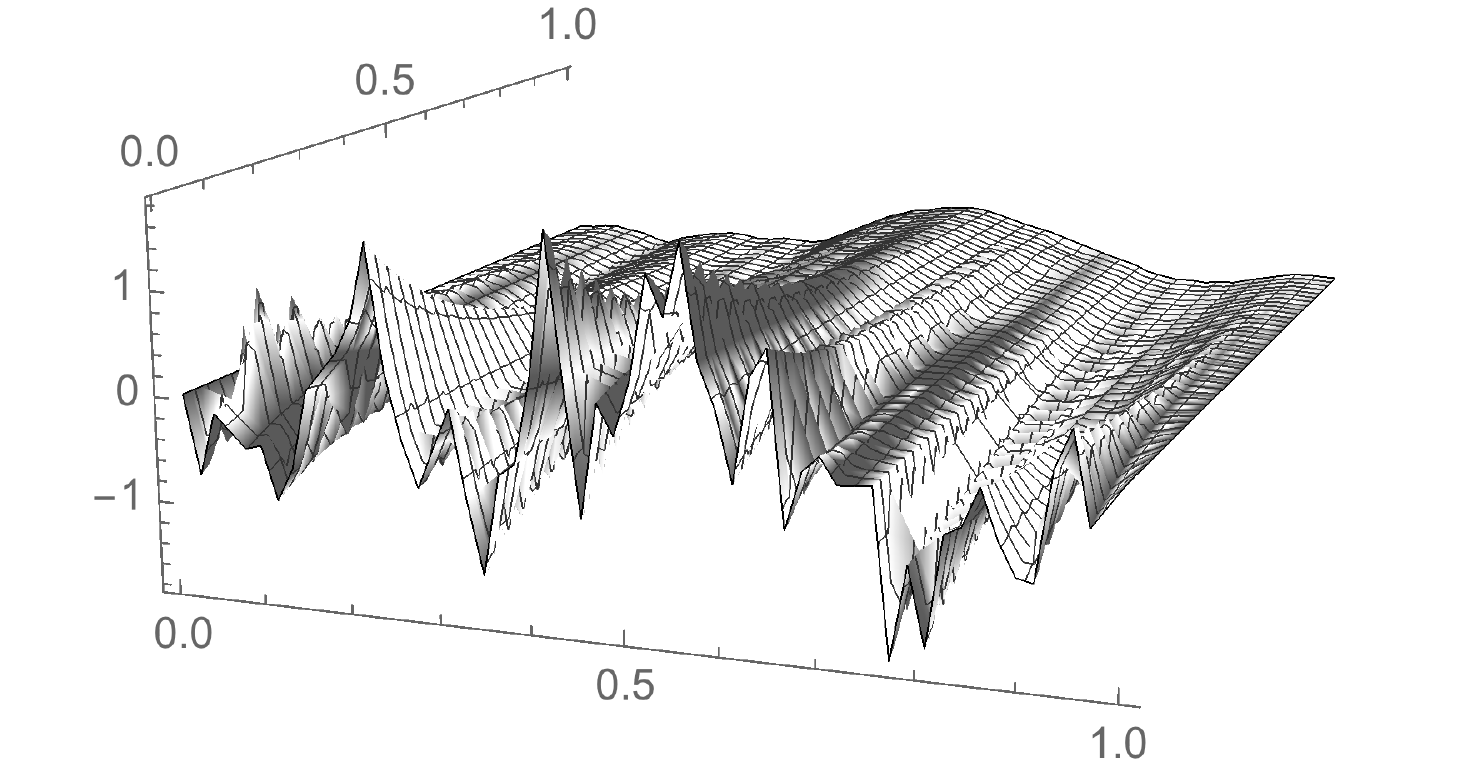}
\put(43,50){$H$}
\put(81,4){$t$}
\end{overpic}
\caption{The function
$(t,H)\longmapsto \sum\limits_{m=0}^{10}2^{m\left(\frac{1}{2}-H\right)}\sum\limits_{k=0}^{2^m-1}\theta_{m,k}e_{m,k}(t)
$
for randomly sampled coefficients $\theta_{m,k}\in\{-1,+1\}$. }
\end{center}\label{XH illustration fig}
\end{figure}

\subsection{The $\mathbf{p^{\text{th}}}$ variation of functions in $\mathbf{\mathfrak X^H}$.}\label{pth variation section}

The signed Takagi--Landsberg functions in the class $\mathfrak{X}^{H}$ share many properties with the sample paths of the corresponding  fractional Brownian bridge with the same Hurst parameter $H\in(0,1)$. For instance, it follows from~\cite[Theorem 3.1 (iii)]{Allaartflexible} that the functions in $\mathfrak{X}^{H}$  are nowhere differentiable. Moreover,~\cite{Baranski} showed that the graph of  $x^H$ has the same Hausdorff dimension, $2-H$, as the  trajectories of fractional Brownian motion~\cite{Orey1970}. In addition, for $H=\frac12$, the functions in $\mathfrak{X}^{H}$ have  linear quadratic variation $\langle x\rangle_t=t$, just as the sample paths of standard Brownian motion or of the Brownian bridge; see~\cite{SchiedTakagi}. Our Theorem~\ref{p-variation thm} will extend the preceding result to all Hurst parameters $H\in(0,1)$. It hence enables  us to use the functions in $\mathfrak X^H$ as test integrators for higher-order  pathwise integration theory in the spirit of F\"ollmer's pathwise It\^o calculus~\cite{FoellmerIto} as it was recently developed by Cont and Perkowski~\cite{ContPerkowski} (see also Gradinaru et al.~\cite{GradinaruEtAl} and Errami and Russo~\cite{ErramiRusso} for related earlier work). 

This integration theory is based on the notion of \emph{continuous $p^{\text{th}}$ variation along a refining sequence of partitions}, which we are going to recall next.
Let \begin{equation}\label{dyadic partition}
\mathbb{T}_n:=\{k2^{-n}\,|\,n\in\mathbb{N}, k=0,\dots, 2^n\},\qquad n=0,1,\dots,
\end{equation}
denote the $n^{\text{th}}$ dyadic partition of $[0,1]$.  It will be convenient to denote by $s'$  the successor of $s$ in $\mathbb{T}_n$, i.e.,
$$s'=\begin{cases}\min\{t\in\mathbb{T}_n\,|\,t>s\}&\text{if $s<1$,}\\
1&\text{if $s=1$.}
\end{cases}
$$
Now let $x$ be a function in $C[0,1]$ and $p\ge1$. The function $x$ admits the \emph{continuous $p^{\text{th}}$ variation $\<x\>_t^{(p)}$ along the sequence $(\mathbb T_n)$}, if for each $t\in[0,1]$ 
$$\<x\>_t^{(p)}:=\lim_{n\uparrow\infty}\sum_{s\in\mathbb T_n,s\le t}|x(s')-x(s)|^p
$$
exists and the function $t\mapsto \<x\>_t^{(p)}$ is continuous (see, e.g.,~\cite[Lemma 1.3]{ContPerkowski}). According to F\"ollmer~\cite{FoellmerIto} in the case $p=2$ and Cont and Perkowski~\cite{ContPerkowski} in the case of general even  $p\in\mathbb N$, this notion of $p^{\text{th}}$ variation along a refining sequence of partitions is the key to a pathwise integration theory with integrator $x$. Note that it is
different from the usual concept of \emph{finite $p$-variation}, which will be discussed at the end of this section.

\begin{theorem}\label{p-variation thm}
Let $H\in(0,1)$, $p>0$, and $x\in\mathfrak X^H$. Then, for all $t\in(0,1]$,
\begin{equation}\label{p-th variation of x}\lim_{n\uparrow\infty}\sum_{s\in\mathbb T_n,s\le t}|x(s')-x(s)|^p=\begin{cases}0&\text{if $p>1/H$,}\\
+\infty&\text{if $p<1/H$,}\\
t\cdot 2^{1-1/H}\mathbb E[|Z_H|^p]&\text{if $p=1/H$,}
\end{cases}
\end{equation}
where
$$Z_H:=\sum_{m=0}^{\infty}2^{m(H-1)}Y_{m}$$
for an  i.i.d.~sequence $Y_0,Y_1,\dots$   of $\{-1,+1\}$-valued random variables with $\mathbb P[Y_n=+1]=\frac12$. 
In particular, for $p=1/H$, each $x\in\mathfrak X^H$ admits the continuous  $p^{\text{th}}$ variation $\<x\>_t^{(p)}=t\cdot 2^{1-1/H} \mathbb E[|Z_H|^p]$ along the sequence $(\mathbb T_n)$ of dyadic partitions. 
\end{theorem}

\begin{remark}Note that the law of $Z_H$ is the infinite Bernoulli convolution with parameter $2^{H-1}$. These laws have been studied in their own right for many decades; see, e.g.,~\cite{PeresSchlagSolomyak}. According to the It\^o-type formulas~\cite[Theorems 1.5 and 1.10]{ContPerkowski}, the most interesting case is the one  in which $p=1/H$ is an even integer. For this case, Theorem 1  from Escribano et al.~\cite{EscribanoEtAl} provides an exact formula for $E[|Z_H|^p]$ in terms of Bernoulli numbers $B_{2k}$ and partitions of $n:=p/2$, namely,
$$\mathbb E[|Z_H|^p]=(-1)^n\!\!\!\!\!\!\!\sum_{1\cdot n_1+\cdots+n\cdot n_n=n}\frac{(2n)!}{n_1!\cdots n_n!}\prod_{k=1}^n\bigg(\frac{1}{(2k)!}\frac{(-1)^k}{2k}2^{2k}(2^{2k}-1)B_{2k}\frac{(1-2^{H-1})^{2k}}{1-2^{2k(H-1)}}\bigg)^{n_k}.
$$ 
If $1/H$ is not an even integer, then \eqref{p-th variation of x} yields
 that any such $x\in\mathfrak X^H$ will have vanishing, and hence continuous, $p^{\text{th}}$ variation along $(\mathbb T_n)$ for  any  $p>1/H$, and so~\cite[Theorems 1.5 and 1.10]{ContPerkowski} can be applied with any even integer $p>1/H$. \end{remark}

\begin{remark}It is known that  fractional Brownian motion $B^H$ with Hurst parameter $H\in(0,1)$ satisfies almost surely,
\begin{equation}\label{fBM variation eq}
\lim_{n\uparrow\infty}\sum_{s\in\mathbb T_n,s\le t}|B^H(s')-B^H(s)|^p=\begin{cases}0&\text{if $p>1/H$,}\\
+\infty&\text{if $p<1/H$,}\\
t\cdot \mathbb E[|B^H(1)|^p]&\text{if $p=1/H$,}
\end{cases}
\end{equation}
for $t>0$; see~\cite[Section 1.18]{Mishura}.
 By Lemma~\ref{additive p-var lemma},  the same is true if, on the left-hand side of \eqref{fBM variation eq}, $B^H$ is replaced by a fractional Brownian bridge. Therefore, our  Theorem~\ref{p-variation thm} establishes yet another similarity between fractional Brownian sample paths and the signed Takagi--Landsberg functions in $\mathfrak X^H$. \end{remark}
\begin{figure}
\begin{center}
\includegraphics[width=10cm]{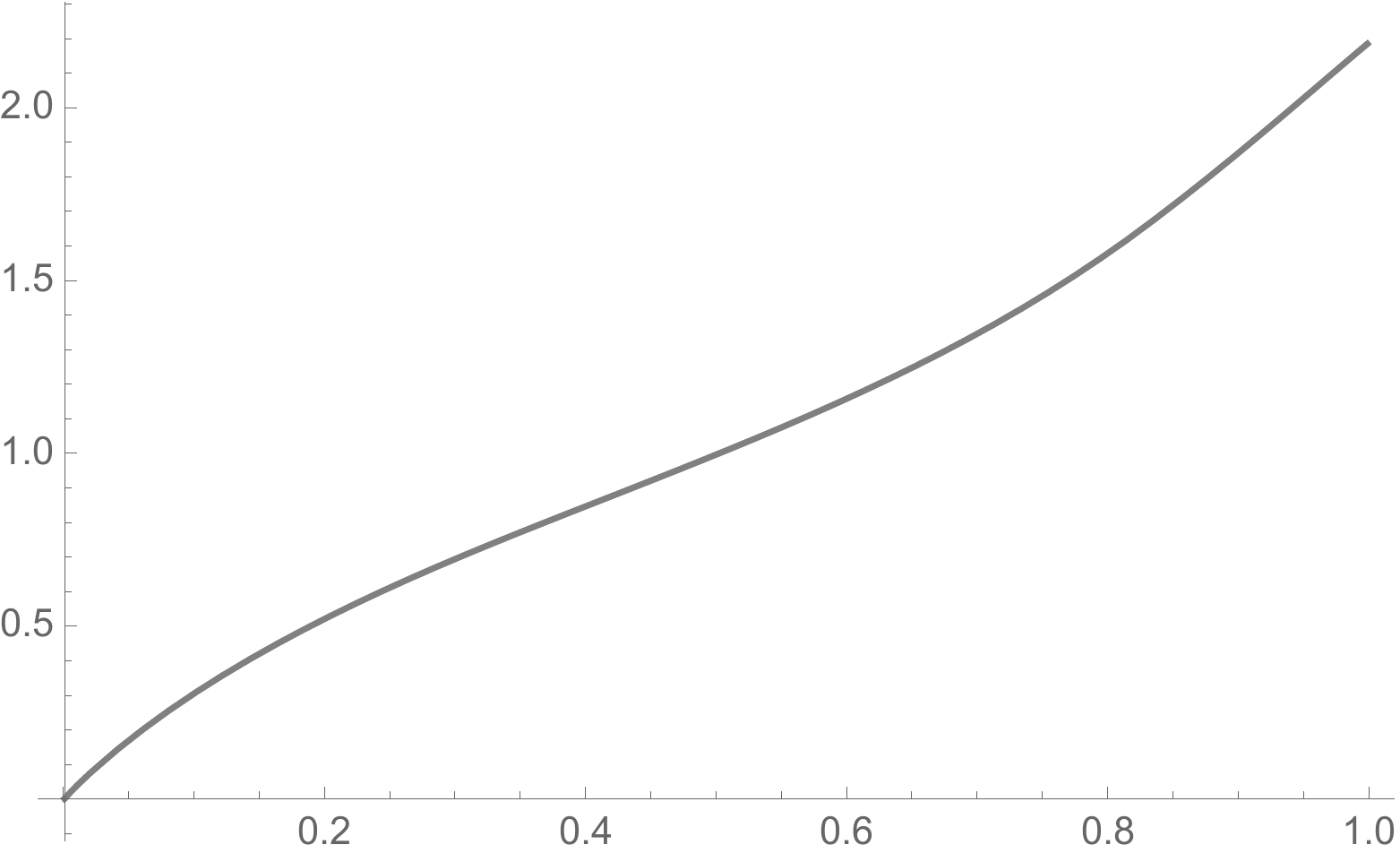}
\end{center}
\caption{The value $\<x^H\>_1^{(1/H)}=2^{1-1/H}\mathbb E[|Z_H|^{1/H}]$ as a function of $H\in(0,1)$.}\label{Figure3}
\end{figure}

 In \cite[Remark 1.7]{ContPerkowski}, it is stated that  an It\^o-type formula with integrator $x\in C[0,1]$ could hold also if $p$ is an odd integer, provided that $x$ admits a continuous $p^{\text{th}}$ variation along $(\mathbb T_n)$ and 
$\sum_{s\in\mathbb T_n,s\le t}(x(s')-x(s))^p$ converges for all $t$ to a continuous function of bounded variation. However,  \cite[Remark 1.7]{ContPerkowski} furthermore states  that 
\lq\lq for odd p we typically expect the limit to be zero", and a corresponding example is provided in the appendix of \cite{ContPerkowski}. The following corollary confirms this claim also in our situation.

\begin{corollary}\label{odd p cor}Suppose that $p=1/H$ is an odd integer and that $x\in\mathfrak{X}^H$. Then, for all $t\in[0,1]$,
$$\lim_{n\uparrow\infty}\sum_{s\in\mathbb T_n,s\le t}(x(s')-x(s))^p=0.
$$
\end{corollary}

\subsection{Maximum and modulus of continuity of (signed) Takagi--Landsberg functions}\label{TL function section}

Our next result concerns the maximum of the Takagi--Landsberg function $x^H$ for $0<H<1$. The maximum of the classical Takagi function corresponding to $H=1$ was obtained by Kahane~\cite{Kahane}, and the case $H=\frac12$ can be deduced  from Lemma 5 in Galkina~\cite{Galkina}, which was later rediscovered by the second author in~\cite[Lemma 3.1]{SchiedTakagi}. The corresponding result  for the maximum of $x^{1/2}$ was stated  independently in Galkin and Galkina~\cite{GalkinGalkina} and in~\cite{SchiedTakagi}.  This, however, does not include the uniqueness of the maximizers $t =\frac13$ and $t=\frac23$, which is an original contribution of the present paper.

\begin{theorem}\label{max thm}
For $0<H<1$, we have
$$\max_{t\in[0,1]}x^H(t)=\frac{1}{3(1-2^{-H})}.
$$
	Moreover, $t =\frac13$ and $t=\frac23$ are the unique points at which the function $x^H(t)$ attains its maximum.\end{theorem}

Note that the maximum points, $t =\frac13$ and $t=\frac23$, are independent of the Hurst parameter $H$ as long as $0<H<1$. This changes at $H=1$, where the maximum is attained at an uncountable Cantor-type set of Hausdorff dimension $\frac12$ (see~\cite{Kahane} or~\cite[Theorem 3.1]{AllaartKawamura}).  For certain cases $H\in(1,2)$, the maximal values of $x^H$ were obtained by Tabor and Tabor \cite{TaborTabor}. For $H\ge 2$, it is easy to see that $x^H$ attains its unique maximum at $t=1/2$; see \cite[Theorem 4]{GalkinGalkina} for details. For $H=1$, the distribution of the maximum of functions in $\mathfrak X^H$ with random coefficients $\theta_{n,k}$ was studied by Allaart \cite{AllaartRandomMax}. Theorem~\ref{max thm} yields the following corollary, which extends~\cite[Theorem 2.2]{SchiedTakagi}, where the particular case $H=1/2$ was treated.

\begin{corollary}\label{osc cor} The uniform maximum of functions  $x\in \mathfrak{ X}^H$ is attained by $x^H$ and equals $$\max_{x\in\mathfrak X^H}\max_{t\in[0,1]}x(t)=\max_{t\in[0,1]}x^H(t)=\frac{1}{3(1-2^{-H})}.$$ The maximal uniform oscillation of functions  $x\in\mathfrak X^H$ is given by
\begin{equation}\label{uniform osc cor eq}
\max_{x\in\mathfrak{ X}^H}\max_{s, t\in[0,1]}|x(t)-x(s)|=\frac{2^H+3}{6(2^H-1)}.
\end{equation}
It is attained for the points $s=\frac13$ and  $t=\frac56$ and for the function 
\begin{equation}\label{wt x}
\wt x^H  =e_{0,0} +\sum_{m=1}^\infty2^{m(\frac12-H)}\left(\sum_{k=0}^{2^{m-1}-1}e_{m,k} -\sum_{k=2^{m-1}}^{2^m-1}e_{m,k} \right)\in\mathfrak X^H.
\end{equation}
\end{corollary}

 Equation \eqref{uniform osc cor eq} will be needed for the proof of our next result, which provides a modulus of continuity for $x^H$. For $h>0$, we let $\nu(h)=\lfloor-\log_2h\rfloor,$ where $\lfloor a\rfloor$ is the biggest integer not exceeding $a$. Then we define 
\begin{equation}\label{omegaH}
\omega_H(h)=\frac{h2^{(\nu(h)-1)(1-H)}}{2^{1-H}-1}+\frac{2^{(1-\nu(h))H} }{3(1-2^{-H})}.
\end{equation}
We have
 $2^{H-1}h^H\leq h2^{\nu(h)(1-H)}\leq h^H$, and $h^H\leq 2^{-\nu(h)H}\leq 2^Hh^H$. It follows that there exists a constant $C>0$ such that 
 \begin{equation}\label{omega=Theta(h)}
 C^{-1}h^H\le\omega_H(h)\le Ch^H\qquad\text{for all sufficiently small $h>0$.}
 \end{equation}
  For $H=\frac12$, the modulus of continuity of $x^H$ was obtained in~\cite[Theorem 2.3 (a)]{SchiedTakagi}. 
  The following theorem extends this result  to all $H\in(0,1)$. Note also that our expression for $\omega_H$ does not extend to the case $H=1$. As a matter of fact, for $H\ge1$, other effects occur and a different method must be used; see~\cite{Kono} and~\cite{Allaartflexible}.

\begin{theorem}\label{xH modulus thm} The following assertions hold.
\begin{enumerate}
\item For all $h\in[0,1)$ and $t\in[0,1-h]$, we have
$|x^H(t+h)-x^H(t)|\le \omega_H(h)$.
\item The inequality in {\rm (a)} is sharp in the sense that 
\begin{align}\limsup_{h\downarrow 0}\max_{t\in[0, 1-h]}\frac{|x^H(t+h)-x^H(t)|}{\omega_H(h)}&=1.\label{modulo 1}
 \end{align}
\end{enumerate}
\end{theorem}

Our next result identifies the uniform modulus of continuity of the functions in the class $\mathfrak X^H$. It extends~\cite[Theorem 2.3 (b)]{SchiedTakagi}, where the special case $H=\frac12$ was treated.

\begin{theorem}\label{modulus thm} The following assertions hold for $\omega_H$ as in \eqref{omegaH}.
\begin{enumerate}
\item For all $h\in[0,1)$ and $t\in[0,1-h]$,
$$\max_{x\in\mathfrak{X}^H}|x(t+h)-x(t)|\le 2^{1-H}\omega_H(h).
$$
\item The inequality from part {\rm(a)} is sharp in the sense that 
\begin{align}
 \limsup_{h\downarrow 0}\max_{t\in[0, 1-h]}\frac{|\wt x^H (t+h)-\wt x^H (t)|}{\omega_H(h)}&=2^{1-H},\label{modulo 2}
 \end{align}
 where $\wt x^H\in\mathfrak X^H$ is as in \eqref{wt x}.
 In particular, the function $2^{1-H}\omega_H(h)$ is a uniform modulus of continuity for the class $\mathfrak{X}^H$,
 \begin{equation}\label{modulo 3}\limsup_{h\downarrow 0}\sup_{x\in\mathfrak{ X}^H}\max_{t\in[0, 1-h]}\frac{|x(t+h)-x(t)|}{2^{1-H}\omega_H(h)}=\limsup_{h\downarrow 0}\max_{t\in[0, 1-h]}\frac{|\wt x^H (t+h)-\wt x^H (t)|}{2^{1-H}\omega_H(h)}=1.  
 \end{equation}
\end{enumerate}
\end{theorem}

\begin{remark}
Theorems~\ref{xH modulus thm} and~\ref{modulus thm} imply in particular that the functions in $\mathfrak X^H$ are H\"older continuous with exponent $H$ but not with any other exponent $\alpha>H$.  This fact already follows from the work of Ciesielski~\cite{Ciesielski}, who showed that the space of H\"older continuous functions is isomorphic to $\ell^\infty$ and the isomorphism consists in a suitable weighting of the coefficients in the Faber--Schauder expansion. We refer to~\cite{GalkinGalkina} for later results that also yield the H\"older continuity of the Takagi--Landsberg function.
The H\"older continuity implies that functions in  $\mathfrak X^H$ are of finite $\frac1H$-variation in the sense of~\cite[Definition 5.1]{FrizVictoir} and can thus be used as test integrators for Young integrals and rough-path calculus; see~\cite{FrizHairer,FrizVictoir} and the references therein. In this respect, the signed Takagi--Landsberg functions in $\mathfrak X^H$ differ from the typical sample paths of fractional Brownian motion with Hurst parameter $H$: It was shown in \cite{Pratelli} that these sample paths have almost surely infinite $\frac1H$-variation in the sense of~\cite[Definition 5.1]{FrizVictoir}.
\end{remark}

\section{Proofs}\label{Proofs}

\subsection{Proofs of Theorem~\ref{p-variation thm}  and Corollary~\ref{odd p cor}}

We will need the following  lemma.

\begin{lemma}\label{additive p-var lemma}Suppose that $p\ge1$ and $x,y\in C[0,1]$ are functions with continuous $p^{\text{th}}$ variation along $(\mathbb T_n)$. Then, if $\<y\>^{(p)}=0$,  the function $x+y$ admits the continuous $p^{\text{th}}$ variation $\<x+y\>^{(p)}=\<x\>^{(p)}$.
\end{lemma}

\begin{proof}For   $z\in C[0,1]$, $t\in[0,1]$, and $n\in\mathbb N$, we define
$$S_{t,n}(z):=\bigg(\sum_{s\in\mathbb T_n,s\le t}|z(s')-z(s)|^p\bigg)^{1/p}.
$$
Minkowski's inequality yields that
\begin{align*}
S_{t,n}(x)-S_{t,n}(y)\le S_{t,n}(x+y)\le S_{t,n}(x)+S_{t,n}(y).
\end{align*}
Passing to the limit $n\uparrow\infty$ thus yields the assertion.
\end{proof}

\begin{proof}[Proof of Theorem~\ref{p-variation thm}] We start by proving the assertion for $p=1/H$ and  $t=1$. For a given function 
\begin{equation}\label{x FS expansion}
x=\sum\limits_{m=0}^{\infty}2^{m\left(\frac{1}{2}-H\right)}\sum\limits_{\ell=0}^{2^m-1}\theta_{m,\ell}e_{m,\ell}\in\mathfrak X^H
\end{equation}
and $n\in\mathbb N$, the corresponding truncated function $x_n$ is defined as
\begin{equation}\label{truncated  function}
x_n=\sum\limits_{m=0}^{n-1}2^{m\left(\frac{1}{2}-H\right)}\sum\limits_{\ell=0}^{2^m-1} \theta_{m,\ell}e_{m,\ell}.
\end{equation}
Then $x(k2^{-n})=x_n(k2^{-n})$ for $k=0,\dots, 2^n$. Moreover, the function $x_n$ is affine on $[k2^{-n},(k+1)2^{-n}]$ with slope
$$\sum_{m=0}^{n-1}2^{m(\frac12-H)}2^{m/2}\sigma_{m,k},
$$
where $\sigma_{m,k}\in\{-1,+1\}$. It follows that 
\begin{equation}\label{binary rep}
x((k+1)2^{-n})-x(k2^{-n})=2^{-n}\sum_{m=0}^{n-1}2^{m(1-H)}\sigma_{m,k}.
\end{equation}
Therefore,
\begin{align}\label{Sn sum}
V_n:=\sum_{k=0}^{2^n-1}\big|x((k+1)2^{-n})-x(k2^{-n})\big|^{p}&=\sum_{k=0}^{2^n-1}\bigg|2^{-n}\sum_{m=0}^{n-1}2^{m(1-H)}\sigma_{m,k}\bigg|^{p}.
\end{align}
Now we establish the convergence of $V_n$ as $n\uparrow\infty$. To this end, we claim that, when summing over $k$ in \eqref{Sn sum}, the vector $(\sigma_{0,k},\dots,\sigma_{n-1,k})^\top$ runs through the entire set $\{-1,+1\}^n$, and each of its elements appears exactly once. This claim will be proved below. Once it has been established,  we can represent $V_n$ by means of an expectation with respect to the uniform distribution on $\{-1,+1\}^n$ and, hence, with respect to 
the random variables $Y_0,\dots, Y_{n-1}$. That is,
\begin{align*}
V_n&=2^n\mathbb E\bigg[\bigg|2^{-n}\sum_{m=0}^{n-1}2^{m(1-H)}Y_{m}\bigg|^{p}\bigg]=\mathbb E\bigg[\bigg|\sum_{m=0}^{n-1}2^{(n-m)(H-1)}Y_{m}\bigg|^{p}\bigg]\\
&=\mathbb E\bigg[\bigg|2^{H-1}\sum_{m=0}^{n-1}2^{m(H-1)}Y_{n-1-m}\bigg|^{p}\bigg]=2^{1-1/H}\mathbb E\bigg[\bigg|\sum_{m=0}^{n-1}2^{m(H-1)}Y_{m}\bigg|^{p}\bigg],
\end{align*}
where we have used our assumption $p=1/H$ in the second step. Clearly, the infinite series
$$\sum_{m=0}^{\infty}2^{m(H-1)}Y_{m}=Z_H$$
converges absolutely, and so 
dominated convergence implies that 
$$\lim_{n\uparrow\infty} V_n= 2^{1-1/H}\mathbb E[|Z_H|^{p}].
$$
This proves the assertion \eqref{p-th variation of x} for $t=1$.

Now we prove our auxiliary claim that when summing over $k$ in \eqref{Sn sum}, the column vector $(\sigma_{0,k},\dots,\sigma_{n-1,k})^\top$ runs through the entire set $\{-1,+1\}^n$, and each of its elements appears exactly once. To this end, we consider first the particular case $x=x^H$. In this case, we write $\sigma^H_{m,k}$ for the signs in \eqref{binary rep}. The row vector $(\sigma^H_{m,0},\dots,\sigma^H_{m,2^n-1})$ consists of alternating blocks of length $2^{n-1-m}$ with entries all being $+1$ or $-1$. We thus obtain the matrix 
\begin{equation}\label{signs matrix}
\begin{pmatrix}\sigma^H_{n-1,0}&\cdots&\sigma^H_{n-1,2^n-1}\\
\sigma^H_{n-2,0}&\cdots&\sigma^H_{n-2,2^n-1}\\
\vdots&\vdots&\vdots\\
\sigma^H_{0,0}&\cdots&\sigma^H_{0,2^n-1}
\end{pmatrix}=\begin{pmatrix}+1&-1&+1&-1&\cdots&\cdots&\cdots&\cdots&+1&-1\\
+1&+1&-1&-1&\cdots&\cdots&\cdots&\cdots&-1&-1\\
\vdots&\vdots&\vdots&\vdots&\vdots&\vdots&\vdots&\vdots&\vdots&\vdots\\
+1&+1&+1&+1&\cdots&+1&-1&\cdots&-1&-1
\end{pmatrix}.
\end{equation}
When replacing all occurrences of $-1$ with $0$,  the columns of this matrix run through the binary representations of all integers from $2^n-1$ to 0. This establishes our auxiliary claim for $x=x^H$. 

Next, we consider the case in which there is exactly one coefficient $\theta_{m,\ell}$ in the representation \eqref{truncated  function} that is equal to $-1$. In this case, only the signs $\sigma_{m,2^{n-m}\ell},\dots,\sigma_{m,2^{n-m}(\ell+1)-1}$ are different from $(\sigma_{i,j}^H)$. More precisely,  
$\sigma_{m,k}=-\sigma^H_{m,k}$ for $k=2^{n-m}\ell,\dots, 2^{n-m}(\ell+1)-1$. That is, the matrix $(\sigma_{i,j})$ is obtained from $(\sigma^H_{i,j})$ by swapping the columns $ 2^{n-m}\ell$ to $(\ell+\frac12)2^{n-m}-1$ with the columns $(\ell+\frac12) 2^{n-m}$ to $(\ell+1)2^{n-m}-1$. In other words, when passing from $(\sigma^H_{i,j})$ to $(\sigma_{i,j})$,  only the order of columns changes and not the entire collection of all columns. Therefore, also for this $x$,  the column vector $(\sigma_{0,k},\dots,\sigma_{n-1,k})^\top$ in \eqref{Sn sum} runs through the entire set $\{-1,+1\}^n$. By iterating this argument and successively introducing further coefficients $\theta_{i,j}=-1$, we obtain our auxiliary claim for all $x\in\mathfrak X^H$. Note that our auxiliary claim implies in particular that $V_n$ in \eqref{Sn sum} does not depend on $x\in\mathfrak X^H$. 

In the next step, we establish the assertion for $p=1/H$  for arbitrary $t\in[0,1]$. To this end, we first recall the following scaling properties of the Faber--Schauder functions:
\begin{equation}\label{emk scaling prop}
\sqrt2 e_{m,k}(t)=e_{m-1,k}(2t)\qquad\text{and}\qquad e_{m,k}(t-\ell 2^{-m})=e_{m,k+\ell}(t)
\end{equation}
for $t\in\mathbb R$, $m\ge0$, and $k,\ell\in\mathbb Z$.  Let $x\in\mathfrak X^H$ have the development \eqref{x FS expansion}. The first scaling property in \eqref{emk scaling prop} implies that for $t\in[0,1/2]$, 
\begin{align*}
x(t)&=\theta_{0,0}e_{0,0}(t)+\sum_{m=1}^\infty 2^{m(\frac12-H)}\sum_{\ell=0}^{2^m-1}\theta_{m,\ell}e_{m,\ell}(t)\\
&=\frac12\theta_{0,0}2t+\sum_{m=1}^\infty 2^{m(\frac12-H)}\sum_{\ell=0}^{2^{m-1}-1}\theta_{m,\ell}e_{m,\ell}(t)\\
&=\frac12\theta_{0,0}2t+2^{-H}\sum_{m=0}^\infty2^{m(\frac12-H)}\sum_{\ell=0}^{2^{m}-1}\theta_{m+1,\ell}e_{m,\ell}(2t).
\end{align*}
That is, there exists $y\in\mathfrak X^H$ and a linear function $f$ such that $x(t)=f(2t)+2^{-H}y(2t)$ for $0\le t\le 1/2$. It is easy to see that our assumption $p>1$ implies that the linear function $f$ has $\<f\>^{(p)}_1=0$. It hence follows from Lemma~\ref{additive p-var lemma} and the assumption $p=1/H$ that 
$$\<x\>_{\frac12}^{(p)}=\<f+2^{-H}y\>_1^{(p)}=\<2^{-H}y\>_1^{(p)}=\frac12\<y\>_1^{(p)}=\frac12\mathbb E[|Z_H|^p].$$
 Iteratively, we obtain $\<x\>^{(p)}_{2^{-k}}=2^{-k}\mathbb E[|Z_H|^p]$ for all $k\in\mathbb N$.  Using also the second scaling property in \eqref{emk scaling prop} gives in a similar way that $\<x\>^{(p)}_{(k+1)2^{-\ell}}-\<x\>^{(p)}_{k2^{-\ell}}=2^{-\ell}\cdot\mathbb E[|Z_H|^p]$ for  $k,\ell\in\mathbb N$ with $(k+1)2^{-\ell}\le1$. We therefore arrive at $\<x\>^{(p)}_t=t \cdot \mathbb E[|Z_H|^p]$ for all dyadic rationals $t$ in $[0,1]$. A sandwich argument extends this fact to all $t\in[0,1]$.

 If $p>1/H$, then 
\begin{align*}
\sum_{s\in\mathbb T_n,s\le t}|x(s')-x(s)|^p&\le \max_{u\in\mathbb T_n}|x(u')-x(u)|^{p-1/H}\sum_{s\in\mathbb T_n,s\le t}|x(s')-x(s)|^{1/H}.
\end{align*}
The maximum tends to zero as $n\uparrow\infty$, due to the uniform continuity of $x$, whereas the sum on the right-hand side converges to $\<x\>^{(1/H)}_t=t\cdot \mathbb E[|Z_H|^{1/H}]<\infty$. Hence, the assertion is proved for  $p>1/H$.

If $t>0$ and $p<1/H$, we  take some $\alpha\in(0,p)$ and let $q:=p/\alpha$ and $r:=p/(p-\alpha)$. Then $\frac1q+\frac1r=1$ and so
\begin{align*}
\sum_{s\in\mathbb T_n,s\le t}|x(s')-x(s)|^{1/H}&=\sum_{s\in\mathbb T_n,s\le t}|x(s')-x(s)|^{\frac1H-\alpha}|x(s')-x(s)|^\alpha\\
&\le \bigg(\sum_{s\in\mathbb T_n,s\le t}|x(s')-x(s)|^{r(\frac1H-\alpha)}\bigg)^{1/r}\bigg(\sum_{s\in\mathbb T_n,s\le t}|x(s')-x(s)|^{p}\bigg)^{1/q}.
\end{align*}
One easily checks that $r(\frac1H-\alpha)>\frac1H$, and so the first sum on the right-hand side tends to zero by the preceding step of the proof. Moreover, the sum on the left converges to the strictly positive value $t\cdot\mathbb E[|Z_H|^{1/H}]$. Therefore, the second sum on the right-hand side must converge to infinity. 
\end{proof}

\begin{proof}[Proof of Corollary~\ref{odd p cor}] As in \eqref{Sn sum}, we get that 
$$
\widetilde V_n:=\sum_{k=0}^{2^n-1}\big(x((k+1)2^{-n})-x(k2^{-n})\big)^{p}=\sum_{k=0}^{2^n-1}\bigg(2^{-n}\sum_{m=0}^{n-1}2^{m(1-H)}\sigma_{m,k}\bigg)^{p}.
$$
Moreover, we have seen in the proof of Theorem~\ref{p-variation thm}  that, when summing over $k$, the vector $(\sigma_{0,k},\dots,\sigma_{n-1,k})^\top$ runs through the entire set $\{-1,+1\}^n$, and each of its elements appears exactly once. Therefore, if $Y_0,Y_1,\dots$ are random variables as in Theorem~\ref{p-variation thm}, then dominated convergence yields that 
$$\widetilde V_n=\red{2^{1-1/H}}\mathbb E\bigg[\bigg(\sum_{m=0}^{n-1}2^{m(H-1)}Y_{m}\bigg)^{p}\bigg]\longrightarrow \red{2^{1-1/H}} \mathbb E\bigg[\bigg(\sum_{m=0}^{\infty}2^{m(H-1)}Y_{m}\bigg)^{p}\bigg]
$$
as $n\uparrow\infty$. Clearly, the rightmost expectation vanishes due to  symmetry reasons. This establishes the assertion for $t=1$. For general $t\in[0,1]$, we can  repeat the scaling arguments used in the proof of Theorem~\ref{p-variation thm}.
\end{proof}

\subsection{Proofs of Theorem~\ref{max thm} and Corollary~\ref{osc cor}}

The following lemma extends Lemma 5 in Galkina~\cite{Galkina}, where the case $H=1/2$ was treated. Lemma 5 in~\cite{Galkina} was later rediscovered by the second author in~\cite[Lemma 3.1]{SchiedTakagi}.

	\begin{lemma}\label{lem1.1}
	Consider the sequence of functions
	$$x^{H}_n(t)=\sum\limits_{m=0}^{n-1}2^{m(\frac{1}{2}-H)}\sum\limits_{k=0}^{2^m-1}e_{m,k}(t),\qquad t\in[0,1], n\geq 1.$$
Define the sequence $$M_n^H=\frac{1}{3(1-2^{-H})}+\frac{(-1)^{n-1}}{3(2^{1-H}+1)2^n}-\frac{2^{-nH}}
{(1+2^{1-H})(2^H-1)},$$
and let $$J_n=\frac{1}{3}(2^n-(-1)^n)$$ be the sequence of Jacobsthal numbers. Then the function $x_n^H$ has exactly two maximal points given by $t_n^-=2^{-n}J_n\in[0,\frac12]$ and $t_n^+=1-t_n^-\in[\frac12,1]$, and the global maximum of $x_n^H$ is given by $$\max_{t\in[0,1]}x_n^H(t)=x_n^H(t_n^-)=x_n^H(t_n^+)=M_n^H.$$
	\end{lemma}

%
	
\begin{proof}
We follow the steps in the proof of~\cite[Lemma 3.1]{SchiedTakagi}, modifying it for our purposes. Note that $x_n^H$ is symmetric with respect to $t=\frac12$.  Therefore, it is sufficient to consider the restriction of $x_n^H$ to $[0,\frac12]$. In what follows, we let  $t_n:=t_n^-.$ Now we proceed by induction on $n$. For $n=1$ we have that  $$x_1^H(t)=e_{0,0}(t)=\min\{t,1-t\},$$
and this function achieves its maximum at $t_1=\frac12=2^{-1}J_1$, the maximal value being equal $x_1^H(t_1)=\frac12=M_1.$
Let $n=2$. Then for $t\in[0,\frac12]$  $$x_2^H(t)=e_{0,0}(t)+2^{\frac12-H}(e_{1,0}(t)+e_{1,1}(t))=e_{0,0}(t)+2^{\frac12-H}e_{1,0}(t).$$
Since $0<H<1$, this function achieves its maximum at $t_2=\frac14=2^{-2}J_2$ and it equals $$x_2^H(t_2)=\frac14+2^{-1-H}=M_2.$$
The maximizers of $x_1^H$ and $x_2^H$ on $[0,\frac{1}{2}]$ are obviously unique. Now, let $n\geq 2$. Note that
$$x_n^H(t)=x_{n-1}^H(t)+2^{(n-1)(\frac12-H)}\sum_{k=0}^{2^{n-1}-1}e_{n-1,k}(t),$$
and
$$x_{n+1}^H(t)=x_{n-1}^H(t)+2^{(n-1)(\frac12-H)}\sum_{k=0}^{2^{n-1}-1}f_{n-1,k}(t),$$
where $$f_{m,k}(t)=e_{m,k}(t)+2^{ \frac12-H }e_{m+1,2k}(t)+2^{ \frac12-H }e_{m+1,2k+1}(t).$$
According to the induction hypothesis, the maximal value of $x_n^H$ is attained at the peak of some function $e_{n-1,k}.$ The support of $f_{n-1,k}$ coincides with the support of $e_{n-1,k}$, and $x_{n-1}^H$ is linear on this support. The function
$f_{n-1,k}$ has two maxima at $t_n-2^{-n-1}$ and $t_n+2^{-n-1}$ and they are strictly larger than the maximum of $e_{n-1,k}$. Therefore, either $x_{n+1}^H(t_n-2^{-n-1})$ or $x_{n+1}^H(t_n+2^{-n-1})$ is strictly larger than $x_n^H(t_n)=\max_{t\in[0,\frac12]}x_n^H(t).$ It means that the maximum of $x_{n+1}^H$ is attained at the peak of some Faber--Schauder function $e_{n,\ell}$ for some index $\ell$.  Let the support of $e_{n,\ell}$ be an interval with endpoints $s$ and $s'$. Then its peak is at point $s^*=\frac{s+s'}{2}$ and has height $2^{-\frac{n}{2}-1}.$ The function $x_n^H$ is linear on the support of $e_{n,\ell}$, therefore,
\begin{align*}\max_{t\in[0,\frac12]}x_{n+1}^H(t)&=x_{n+1}^H(s^*)=x_{n}^H(s^*)+2^{n(\frac12-H)}e_{n,\ell}(s^*)\\
&=
\frac{x_{n}^H(s)+x_{n}^H(s')}{2}+2^{n(\frac12-H)}2^{-\frac{n}{2}-1}=\frac{x_{n}^H(s)+x_{n}^H(s')}{2}+2^{-nH-1}.
\end{align*}
At one of the points $s$ or $s'$, let it be $s'$,  the function $x_{n}^H$ coincides with $x_{n-1}^H$. Therefore, $x_{n}^H(s')\leq M_{n-1}^H.$ Obviously, $x_{n}^H(s)\leq M_{n}^H.$ Therefore,
\begin{equation}\label{bound-above}
 x_{n+1}^H(s^*)\leq \frac{M_n^H+M_{n-1}^H}{2}+2^{-nH-1}=M_{n+1}^H.
\end{equation}
We can take $s^*$ as the midpoint between $t_n$ and $t_{n-1}$, because  $t_n$ and $t_{n-1}$ enclose the interval of the support of some Faber--Schauder function of $n$th generation. In this case, we have equality in \eqref{bound-above}, and, moreover, $s^*=\frac{t_n+t_{n-1}}{2}=t_{n+1}.$
We immediately conclude that $t_{n+1}$ is the unique maximizer of $x_{n+1}^H$ in $[0,\frac12]$ because by induction hypothesis, $\frac{t_n+t_{n-1}}{2}$  is the only point at which we have equality in \eqref{bound-above}.
\end{proof}

	\begin{proof}[Proof of Theorem~\ref{max thm}]  By sending $n$ to infinity  in Lemma~\ref{lem1.1},  one easily gets our formula for the maximum of $x^H$ and the fact that $x^H$ is maximized at $\lim_n t_n^-=\frac13$ and $\lim_nt_n^+=\frac23$. It thus remains to show that these are the only maximum points. Since $x^H(t)$ is symmetric around $t=1/2$ by \eqref{emk scaling prop}, we can concentrate on the restriction of $x^H$ to the interval $[0,\frac12]$.  Let us thus assume by way of contradiction that there exists $t_0\in[0,\frac12]\setminus\{\frac13\}$ such that $x^H(t_0)=x^H(\frac13)$. Let $t_n:=t_n^-=2^{-n}J_n$ be as in Lemma~\ref{lem1.1}. Then, for all $n$, the limit $\frac13$ is contained in the dyadic interval $I_n$ with endpoints $t_n$ and $t_{n-1}$. 
	Let similarly $S_n$ be that dyadic interval of the form $[k2^{-n},(k+1)2^{-n}]$ such that $t_0\in S_n$. Then there exists  $m\in\mathbb N$ such that $I_n$ and $S_n$ are disjoint for all $n\ge m$. Note that $x^H_m$ is linear on each of the two intervals $I_m$ and  $S_m$. 
	
	Next, we let $s_{0}$ and $s_1$ be the endpoints of $S_m$,  ordered such that $x_m^H(s_0)\le x_m^H(s_1)$. Since $s_0$ and $s_1$ are neighboring points in the $m^{\text{th}}$ dyadic partition $\mathbb T_m$, there must be $i\in\{0,1\}$ such that $s_i\in\mathbb T_{m-1}$. Since $t_{m-1}$ and $t_m$ are the unique respective maximizers of $x_{m-1}^H$ and $x_m^H$  in $[0,\frac12]$ 
	and $t_{m-1},t_m\notin S_m$, we have
		\begin{equation}\label{si eq}	x_m^H(s_i)=x_{m-1}^H(s_i)<x^H_{m-1}(t_{m-1})=x^H_{m}(t_{m-1})\qquad\text{and}\qquad x_m^H(s_{1-i})<x_m^H(t_m).
	\end{equation}
	Now let $f:[0,2^{-m}]\to \mathbb R_+$ be the function that increases linearly from $x_m^H(t_{m-1})$ to $x_m^H(t_{m})$, and  we define 
	$g:[0,2^{-m}]\to \mathbb R_+$ as the function that increases linearly from $x_m^H(s_0)$ to $ x_m^H(s_1)$. 	By considering the two possibilities $i\in\{0,1\}$, we get from \eqref{si eq} that 	$g(t)<f(t)$ for all $t\in[0,2^{-m}]$. 
	Moreover, our assumption that both $t_0$ and $\frac13$ are maximizers of $x^H$, together with the scaling properties \eqref{emk scaling prop} and the linearity of $x^H_m$ on both $S_m$ and $I_m$, implies that 	for $\widetilde t_0:=t_0-\min\{s_0,s_1\}\in[0,2^{-m}]$,
			\begin{align*}
\max_{r\in[0,1]}x^H(r)&=	x^H(t_0)\\
&=g(\widetilde t_0)+\sum_{n=m}^\infty2^{n(1/2-H)}\sum_{k=0}^{2^n-1}e_{n,k}(\widetilde t_0)\\
&<f(\widetilde t_0)+\sum_{n=m}^\infty2^{n(1/2-H)}\sum_{k=0}^{2^n-1}e_{n,k}(\widetilde t_0)\\
&\le \max_{t\in I_n}x^H(t)=x^H(1/3)=\max_{r\in[0,1]}x^H(r).
	\end{align*}
This is the desired contradiction.	\end{proof}

\begin{proof}[Proof of Corollary~\ref{osc cor}] 
The inequality
\begin{equation}\label{ineq1.2}\max_{s,t\in[0,1]}|x(s)-x(t)|\leq \max_{s\in[0,1/2]}\max_{t\in[1/2,1]}|\wt x^H (s)-\wt x^H (t)|,\end{equation}
for any $x\in \mathfrak{ X}^H$  can be proved in the same way as (3.9) in~\cite{SchiedTakagi}.  Taking $x=\wt x^H$ thus yields that the right-hand side of \eqref{ineq1.2} is equal to $\max_{s,t\in[0,1]}|\wt x^H (s)-\wt x^H (t)|$. Therefore, 
$$\max_{x\in\mathfrak X^H}\max_{s,t\in[0,1]}|x(s)-x(t)|=\max_{s,t\in[0,1]}|\wt x^H (s)-\wt x^H (t)|.
$$
Furthermore, $x^H=\wt x^H $ on $[0,\frac12]$ so the maximal value of  $\wt x^H $ is attained at $t_1=\frac13$ and equals $\frac{1}{3(1-2^{-H})}.$ On the interval $[\frac12, 1]$ we have that $\wt x^H (t)=\frac12-x^{H}(t-\frac12), $ whence the minimal value of $\wt x^H $ is achieved
at $t_2=\frac56$ and equals $\frac12-\frac{1}{3(1-2^{-H})}$. From here,  the result follows.
\end{proof}

\subsection{Proofs of Theorem~\ref{xH modulus thm} and~\ref{modulus thm}}

\begin{proof}[Proof of Theorem~\ref{xH modulus thm}] The proof extends arguments from the proof of~\cite[Theorem 3.1 (a)]{SchiedTakagi}. Let $h>0$ be given and $n=\nu(h)$. Then $2^{-n-1}<h\leq 2^{-n}.$ Since the Faber--Schauder functions $e_{m,k}$ are linear with the slope $\pm2^{\frac m2}$ on the dyadic intervals of length $2^{-m-1}$, we get for $m\leq n-2$ and fixed $t\in[0,1-2^{-n})$  that
\begin{equation}\begin{split}\label{eqref 3.09} |x^H(t+h)-x^H(t)|&\leq \sum_{m=0}^{n-2} 2^{m(\frac12-H)}h2^{\frac m2}
+\left|\sum_{m=n-1}^\infty 2^{m(\frac12-H)}\sum_{k=0}^{2^m-1}(e_{m,k}(t+h)-e_{m,k}(t))\right|.\end{split}\end{equation}
The first sum on the right-hand side equals
\begin{equation}\label{eqref 3.09b}
\frac{2^{(n-1)(1-H)}-1}{2^{1-H}-1}h.
\end{equation}
To analyze the second sum, let the integer  $\ell$ be such that $\ell2^{-n}\le t<(\ell+1)2^{-n}$.  Then  $s:=t-\ell 2^{-n}$ is such that  $ 0\le  2^{n-1}s<1/2$ and $ 1/4<2^{n-1}(s+h)<1$. Hence, the scaling properties \eqref{emk scaling prop}
 imply that for $m\ge n$,
 \begin{align}\label{Faber-Schauder rescaling}
 e_{m,k}(t+h)-e_{m,k}(t)=2^{\frac{1-n}2}\big(e_{m-(n-1), k-\ell2^{m-n}}( 2^{n-1}(s+h))-e_{m-(n-1), k-\ell2^{m-n}}( 2^{n-1}s)\big).
 \end{align}
If $\ell$ is even, the preceding identity also holds for $m=n-1$, and we arrive at 
\begin{equation}\label{ell even case}
\begin{split}
\bigg|\sum_{m=n-1}^\infty 2^{m(\frac12-H)}\sum_{k=0}^{2^m-1}(e_{m,k}(t+h)-e_{m,k}(t))\bigg|&=2^{(1-n)H}\big|x^H(2^{n-1}s+h))-x^H(2^{n-1}s)\big|\\
&\le 2^{(1-n)H}\frac{1}{3(1-2^{-H})},
\end{split}
\end{equation}
where we have used Theorem~\ref{max thm} in the second step.  It now follows from \eqref{eqref 3.09} and \eqref{eqref 3.09b} that 
 \begin{equation}\label{xH estimate mod}\begin{split}|x^H(t+h)-x^H(t)|&\leq \frac{2^{(n-1)(1-H)}-1}{2^{1-H}-1}h+2^{(1-n)H }\frac{1}{3(1-2^{-H})}\leq \omega_H(h). \end{split}\end{equation}
If $\ell$ is odd, then the fact that
$$e_{n-1,(\ell-1)/2}(r)+e_{n-1,(\ell+1)/2}(r)=2^{\frac{1-n}2}\Big(\frac12-e_{0,0}(2^{{n-1}}r-\ell/2)\Big),\qquad\text{$r\in[\ell2^{-n},(\ell+2)2^{-n}]$,}
$$
yields that 
\begin{align*}
\bigg|\sum_{m=n-1}^\infty 2^{m(\frac12-H)}\sum_{k=0}^{2^m-1}(e_{m,k}(t+h)-e_{m,k}(t))\bigg|&=2^{(1-n)H}\big|
y(2^{n-1}(s+h)-y(2^{n-1}s)\big|,
\end{align*}
where, for $\varphi(u)=|1/2-u|$, the function $y$ is given by
\begin{equation}\label{y fct}
y =-e_{0,0}+\sum_{m=1}^\infty 2^{m(\frac12-H)}\sum_{k=0}^{2^m-1}e_{m,k}=x^H\circ \varphi-\frac12.
\end{equation}
 Thus, \eqref{xH estimate mod} holds also in case $\ell$ is odd. This concludes the proof of part (a).

To establish part (b), put $t=0$ and $h_n=\frac 23 2^{-n}.$  Then 
 \begin{equation*}\begin{split}|x^H(t+h_n)-x^H(t)|&=x^H( h_n)=\frac{2^{(n-1)(1-H)}-1}{2^{1-H}-1}h_n+2^{(1-n)H }\frac{1}{3(1-2^{-H})}\\&= \omega_H(h_n)-\frac{h_n}{2^{1-H}-1}. \end{split}\end{equation*}
It hence follows  from \eqref{omega=Theta(h)}  that $$\limsup_{n\uparrow\infty}\frac{|x^H(h_n)-x^H(0)|}{\omega_H(h_n)}\geq 1,$$ and the proof of part (b) is complete.\end{proof}

\begin{proof}[Proof of Theorem~\ref{modulus thm}] The proof extends arguments from the proof of~\cite[Theorem 3.1 (b)]{SchiedTakagi}.  To prove part (a), let $h>0$ be given, and $n=\nu(h)$. Arguing as in \eqref{eqref 3.09} and \eqref{eqref 3.09b},  we get for $m\leq n-2$ and for  $t\in[0,1-2^{-n})$  that
\begin{equation}\begin{gathered}\label{eqref 3.11} |x(t+h)-x(t)|\leq \frac{2^{(n-1)(1-H)}-1}{2^{1-H}-1}h
+\left|\sum_{m=n-1}^\infty 2^{m(\frac12-H)}\sum_{k=0}^{2^m-1}\theta_{m,k}(e_{m,k}(t+h)-e_{m,k}(t))\right|.\end{gathered}\end{equation}
To analyze the sum on the right-hand side,  let the integer  $\ell$ be such that $\ell2^{-n}\le t<(\ell+1)2^{-n}$.  Then  $s:=t-\ell 2^{-n}$ is such that  $ 0\le  2^{n-1}s<1/2$ and $ 1/4<2^{n-1}(s+h)<1$. We distinguish three cases. 
 \begin{enumerate}[{\rm (i)}]
 \item  If $\ell $ is even, then \eqref{Faber-Schauder rescaling} yields that
 \begin{equation}\label{even L} \sum_{m=n-1}^\infty 2^{m(\frac12-H)}\sum_{k=0}^{2^m-1}\theta_{m,k}(e_{m,k}(t+h)-e_{m,k}(t))=2^{(1-n)H }(y(2^{n-1}(s+h)-y(2^{n-1}s)), \end{equation}
 where    $$y=\sum_{m=0}^\infty 2^{m(\frac12-H)}\sum_{k=0}^{2^m-1}\theta_{m+n-1,k+\ell 2^{m-1}}e_{m,k}\in \mathfrak{ X}^H.$$
\item  If $\ell $ is odd and, additionally, $\theta_{n-1, \frac{\ell -1}{2}}=\theta_{n-1, \frac{\ell +1}{2}}$, then \eqref{even L} holds with $$y=-\theta_{n-1, \frac{\ell -1}{2}}e_{0,0}+\sum_{m=1}^\infty 2^{m(\frac12-H)}\sum_{k=0}^{2^m-1}\theta_{m+n-1,k+\ell 2^{m-1}} e_{m,k}\in \mathfrak{ X}^H.$$
 \item  If $\ell $ is odd  and  $\theta_{n-1, \frac{\ell -1}{2}}=-\theta_{n-1, \frac{\ell +1}{2}}$, then, similarly to \eqref{eqref 3.11} and~\cite[Eq.~(3.15)]{SchiedTakagi},
     \begin{equation}\bigg|\sum_{m=n}^\infty 2^{m(\frac12-H)}\sum_{k=0}^{2^m-1}\theta_{m,k}(e_{m,k}(t+h)-e_{m,k}(t))\bigg|=2^{-nH}\big|y(2^n(s+h)))-y(2^ns)\big|,
\end{equation}
where the function $y$ can take nonzero values on $[0,2]$ and can be decomposed as $y(r)=y_1(r)\Ind{[0,1]}(r)+y_2(r-1)\Ind{[1,2]}(r)$, for certain functions $y_1,y_2\in \mathfrak{ X}^H $. The exact expressions of the functions $y_i$ are not important  for our further considerations.
\end{enumerate}

Now we analyze the right-had side of the inequality \eqref{eqref 3.11} according to the cases (i)--(iii). In the cases (i) and (ii), we get as in \eqref{ell even case} and by using Corollary \ref{osc cor} that
\begin{align}
 |x(t+h)-x(t)|&\leq \frac{2^{(n-1)(1-H)}-1}{2^{1-H}-1}h
+2^{(1-n)H}\sup_{y\in\mathfrak X^H}\big|y(2^{n-1}(s+h)-y(2^{n-1}s)\big|\nonumber\\
&\le  \frac{2^{(n-1)(1-H)}}{2^{1-H}-1}h
+2^{(1-n)H}\frac{2^H+3}{6(2^H-1)}\nonumber\\
&=2^{1-H} \frac{2^{(n-1)(1-H)}}{2^{1-H}-1}h+(1-2^{1-H})\frac{2^{(n-1)(1-H)+\log_2h}}{2^{1-H}-1}+2^{(1-n)H}\frac{2^H+3}{6(2^H-1)}\nonumber\\
&\le 2^{1-H}  \frac{2^{(n-1)(1-H)}}{2^{1-H}-1}h+2^{(1-n)H}\bigg(\frac{2^H+3}{6(2^H-1)}-\frac14\bigg)\label{omega estimate eq}.
\end{align}
After some manipulations, we find that $$2^{(1-n)H}\bigg(\frac{2^H+3}{6(2^H-1)}-\frac14\bigg)=\frac34\Big(\frac32-\frac{2^H}{6}\Big)\cdot 2^{1-H}\frac{2^{(1-n)H}}{3(1-2^{-H})}.
$$
Since $\frac34(\frac32-\frac{2^H}{6})<1$ for $H>0$, we conclude that \eqref{omega estimate eq} is bounded from above by $2^{1-H}\omega_H(h)$, as desired. 

In case (iii), we get 
\begin{align*}
 |x(t+h)-x(t)|&\leq \frac{2^{(n-1)(1-H)}-1}{2^{1-H}-1}h
+2^{-nH}\sup_{r,s\in[0,2]}|y(r)-y(s)|,
\end{align*}
where $y(r)=y_1(r)\Ind{[0,1]}(r)+y_2(r-1)\Ind{[1,2]}(r)$, for certain functions $y_1,y_2\in \mathfrak{ X}^H $. The supremum on the  right-hand side  realized when $y_1=-y_2=x^H$ and,  according to Theorem~\ref{max thm}, given by $2/(3(1-2^{-H}))$. Therefore, 
$$ |x(t+h)-x(t)|\leq \frac{2^{(n-1)(1-H)}-1}{2^{1-H}-1}h
+\frac{2^{1-nH}}{3(1-2^{-H})}\le 2^{1-H}\omega_H(h).
$$
This completes the proof of part (a).
The proof of part (b) can be completed as the one of~\cite[Theorem 3.1 (b)]{SchiedTakagi}.
\end{proof}

\noindent{\bf Acknowledgement:} The authors are grateful to Zhenyuan Zhang for pointing out that, in the published version of this paper,  the factor $2^{1-1/H}$ is missing in \eqref{p-th variation of x} and in the caption of Figure \ref{Figure3}.
	 \parskip-0.5em\renewcommand{\baselinestretch}{0.9}\normalsize
\bibliography{CTbook}{}
\bibliographystyle{abbrv}
\end{document}